\documentclass[reqno]{amsart}
\usepackage{amssymb, amsmath, amsthm, enumerate, mathtools, graphicx, enumitem, tikz, color, soul, hyperref}
\usepackage[mathscr]{euscript}
\allowdisplaybreaks

\newcommand{\fH}{\mathcal{H}}
\newcommand{\fE}{\mathcal{E}}

\newcommand{\fT}{\mathcal{T}}
\newcommand{\fI}{\mathcal{I}}

\newcommand{\R}{\mathbb{R}}

\newcommand{\Z}{\mathbb{Z}}

\newcommand{\fK}{\mathcal{K}}
\newcommand{\fC}{\mathcal{C}}

\newcommand{\rC}{\operatorname{C}}

\newtheorem{theorem}{Theorem}[section]
\newtheorem{proposition}[theorem]{Proposition}
\newtheorem{lemma}[theorem]{Lemma}

\theoremstyle{definition}
\newtheorem{definition}[theorem]{Definition}

\newtheorem{remark}[theorem]{Remark}

\numberwithin{equation}{section}

\begin{document}
\title[restriction to a hyperbolic cone]{Fourier restriction to a hyperbolic cone}
\author{Benjamin Baker Bruce}
\address{Department of Mathematics, University of Wisconsin, Madison}
\email{bbruce@math.wisc.edu}
\date{\today}
\begin{abstract}
Using a bilinear restriction theorem of Lee and a bilinear-to-linear argument of Stovall, we obtain the conjectured range of Fourier restriction estimates for a conical hypersurface in $\R^4$ with hyperbolic cross sections.
\end{abstract}
\maketitle

\section{Introduction}

In this article, we resolve the Fourier restriction problem for the conical hypersurface
\begin{align*}
\Gamma := \bigg\{\bigg(\zeta,\sigma,\frac{\zeta_1\zeta_2}{\sigma}\bigg) : \zeta \in [-1,1]^2,~\sigma \in [1,2]\bigg\}
\end{align*}
in $\R^4$. In this case, the problem asks, for which exponents $p,q$ is the extension (adjoint restriction) operator
\begin{align*}
\fE f(x,x',t) := \iint_{[-1,1]^2 \times [1,2]} e^{i(x,x',t)\cdot(\zeta,\sigma,\frac{\zeta_1\zeta_2}{\sigma})}f(\zeta,\sigma)d\zeta d\sigma
\end{align*}
of strong type $(p,2q)$?  The restriction problem for the light cone in $\R^4$ was solved by Wolff \cite{Wolff}, while for other conical hypersurfaces, such as those with negatively curved cross sections, it has remained open.  In the case of $\Gamma$, nearly optimal results are known:  Greenleaf \cite{Greenleaf} proved that $\fE$ is of strong type $(p,2q)$ for $p \geq q'$ and $q \geq 2$, and Lee \cite{Lee} extended that range to $q > 3/2$ and $p > q'$.  The main result of this article is the boundedness of $\fE$ on the scaling line  $p = q'$ for $3/2 < q < 2$, solving the remaining part of the restriction problem for $\Gamma$.

\begin{theorem}\label{thm1} The operator $\fE$ is of strong type $(q',2q)$ for $3/2 < q < 2$.
\end{theorem}

Because $\Gamma$ is (a compact piece of) a cone whose cross sections are hyperbolic paraboloids, the slicing argument in \cite{Nicola} shows that a strong type $(p,2q)$ restriction estimate for the hyperbolic paraboloid in $\R^3$ implies the corresponding result for $\Gamma$.  Therefore, by \cite{Stovall} (and the references therein), the estimate in Theorem \ref{thm1} is known for $q > 13/8$ and holds conditionally for smaller $q$, pending further progress on restriction to the hyperbolic paraboloid.  The superior bilinear restriction theory for $\Gamma$, in relation to that of the hyperbolic paraboloid, allows us to prove Theorem \ref{thm1} unconditionally.

\medskip
\noindent {\bf Terminology and notation.}  A positive constant is \emph{admissible} if it depends only on $q$.  We write $A \lesssim B$ or $A = O(B)$ to mean $A \leq CB$ for some admissible constant $C$, which is allowed to change from line to line.  We denote the one-dimensional Hausdorff measure by $\fH^1$.  We write $\log$ for the base $2$ logarithm.  An interval of the form $[n2^{-j},(n+1)2^{-j})$ for some $j,n \in \Z$ is \emph{dyadic}, and $\fI_j$ denotes the set of dyadic intervals of length $2^{-j}$.  The product of two dyadic intervals is a \emph{tile}, and $\fT_{j,k}$ denotes the set of $2^{-j} \times 2^{-k}$ tiles.  Given $\tau \in \fT_{j,k}$, we set $\tilde{\tau} := \tau \times [1,2]$.  We denote by $\pi_{i,3}$ and $\pi_1$, respectively, the projections $(\zeta,\sigma) \mapsto (\zeta_i,\sigma)$ and $(\zeta_i,\sigma) \mapsto \zeta_i$, for $i =1,2$ and $(\zeta,\sigma) \in \R^2 \times \R$.  If $\pi$ is one of these projections and $S$ a subset of the domain of $\pi$, the \emph{$\pi$-projection of $S$} refers to the set $\pi(S)$, and a \emph{$\pi$-fiber of $S$} is any set of the form $\pi^{-1}(\pi(s)) \cap S$ with $s \in S$.  \emph{Horizontal} and \emph{vertical} refer to the directions in $\R^2$ parallel to the standard basis vectors ${\mathbf e_1}$ and ${\mathbf e_2}$, respectively.  Finally, the \emph{extension of a set} refers to the Fourier extension of the set's characteristic function.

\medskip
\noindent {\bf Outline of the proof.}  We adapt an argument of Stovall \cite{Stovall} which showed that, for $3/2 < q < 2$, the extension operator associated to the hyperbolic paraboloid in $\R^3$ is of strong type $(q',2q)$, provided an appropriate $L^{p_0} \times L^{p_0} \rightarrow L^{q_0}$ bilinear restriction inequality holds for some $q_0 < q$ and $p_0/2 < q_0 < p_0'$.  A bilinear estimate suitable for running Stovall's argument on the hypersurface $\Gamma$ is already known:

\begin{theorem}[Lee \cite{Lee}]\label{thm2}
Let $\tau,\kappa \subseteq [-1,1]^2$ be squares with unit separation in both the horizontal and vertical directions.  If $q > 3/2$, then
\begin{align*}
\|\fE f\fE g\|_q \lesssim \|f\|_2\|g\|_2
\end{align*}
for all bounded measurable functions $f$ and $g$ supported in $\tau \times [1,2]$ and $\kappa \times [1,2]$, respectively.
\end{theorem}

To prove Theorem \ref{thm1}, it suffices to show that $\fE$ is of restricted strong type $(q',2q)$ for every $3/2 < q < 2$.  Thus, we aim to prove that
\begin{align}\label{ex1}
\|\fE \chi_\Omega\|_{2q} \lesssim |\Omega|^\frac{1}{q'}
\end{align}
for an arbitrary measurable set $\Omega \subseteq [-1,1]^2 \times [1,2]$.  In Section \ref{sec2}, we use Theorem \ref{thm2} and a bilinear-to-linear argument of Vargas \cite{Vargas} to show that sets having roughly constant $\pi_{1,3}$- (or $\pi_{2,3}$-) fiber length obey \eqref{ex1}.  In Section \ref{sec3}, we solve a related inverse problem:  For which sets $\Omega$ of constant fiber length can the inequality in \eqref{ex1} be reversed?  Oversimplified, our answer is that $\Omega$ must be a box of the form $\tilde{\tau}$; proving \eqref{ex1} then becomes a matter of bounding the extension of a union of boxes, which we do in Section \ref{sec4}.  Our real answer, however, is quantitative:  We show that $\Omega$ is approximately a union of boxes, where the number of boxes in the union and the tightness of the approximation are controlled by the constant $\rC(\Omega)$, defined thus:
\begin{definition}
For measurable sets $\Omega_1 \subseteq \Omega_2 \subseteq [-1,1]^2 \times [1,2]$, let $\rC(\Omega_1,\Omega_2)$ denote the smallest number $\varepsilon$, either dyadic, zero, or infinite, such that $\|\fE\chi_{\Omega_1'}\|_{2q} \leq \varepsilon|\Omega_2|^{1/q'}$ for every measurable set $\Omega_1' \subseteq \Omega_1$, and let $\rC(\Omega_1) := \rC(\Omega_1,\Omega_1)$.
\end{definition}
\noindent Finally, in Section \ref{sec5}, we start with a generic set $\Omega$, decompose it into sets $\Omega(K)$ of fiber length roughly $2^{-K}$, sorted thence according to the value of $\rC(\Omega(K))$, and apply the restriction estimates of Sections \ref{sec3} and \ref{sec4} to obtain \eqref{ex1}.

While much of our argument resembles Stovall's in \cite{Stovall}, we include full details for the convenience of the reader.

\medskip
\noindent {\bf Acknowledgments.}  The author thanks Sanghyuk Lee for introducing him to this problem and Betsy Stovall for her advice.  This work was supported by National Science Foundation grants DMS-1653264 and DMS-1147523.

\section{Extensions of sets of constant fiber length}\label{sec2}
In this section, we prove a scaling line restriction estimate for characteristic functions of sets of constant $\pi_{1,3}$-fiber length, arguing \`a la Vargas \cite{Vargas}.  By symmetry, the same estimate then holds for sets of constant $\pi_{2,3}$-fiber length.
\begin{definition}\label{def2}
Given a measurable set $\Omega \subseteq [-1,1]^2 \times [1,2]$ and an integer $K \geq 0$, let
\begin{align*}
\Omega(K) := \{(\zeta,\sigma) \in \Omega : \fH^1(\pi_{1,3}^{-1}(\zeta_1,\sigma) \cap \Omega) \sim 2^{-K}\}.
\end{align*}
\end{definition}

\begin{proposition}\label{prop1}
Suppose that $\Omega = \Omega(K)$ for some $K$.  Then $\rC(\Omega) \lesssim 1$.
\end{proposition}
\begin{proof}
Let $\Omega' \subseteq \Omega$ be measurable. Given $\tau,\kappa \in \fT_{j,k}$, we write $\tau \sim \kappa$ if $\tau$ and $\kappa$ are separated by a distance of $\sim 2^{-j}$ in the horizontal direction and $\sim 2^{-k}$ in the vertical direction.  Up to a set of measure zero, we have
\begin{align*}
([-1,1]^2 \times [1,2])^2 = \bigcup_{j,k}\bigcup_{\substack{\tau,\kappa \in \fT_{j,k}\\ \tau \sim \kappa}} \tilde{\tau} \times \tilde{\kappa}.
\end{align*}
Consequently, by the triangle inequality and Lemma 6.1 in \cite{TVV} (using that $q < 2$),
\begin{align*}
\|\fE \chi_{\Omega'}\|_{2q}^2 \lesssim \sum_{j,k}\bigg(\sum_{\substack{\tau,\kappa \in \fT_{j,k}\\ \tau \sim \kappa}} \|\fE \chi_{\Omega' \cap \tilde{\tau}} \fE\chi_{\Omega' \cap \tilde{\kappa}}\|_q^q\bigg)^\frac{1}{q}.
\end{align*}
By rescaling, Theorem \ref{thm2} implies that
\begin{align*}
\|\fE \chi_{\Omega' \cap \tilde{\tau}} \fE \chi_{\Omega' \cap \tilde{\kappa}}\|_q \lesssim 2^{-(j+k)(1-\frac{2}{q})}|\Omega' \cap \tilde{\tau}|^\frac{1}{2}|\Omega' \cap \tilde{\kappa}|^\frac{1}{2} \leq 2^{-(j+k)(1-\frac{2}{q})}|\Omega \cap \tilde{\tau}|^\frac{1}{2}|\Omega \cap \tilde{\kappa}|^\frac{1}{2}
\end{align*}
for $\tau,\kappa \in \fT_{j,k}$ with $\tau \sim \kappa$.  Given $\tau \in \fT_{j,k}$, there are admissibly many $\kappa$ such that $\tau \sim \kappa$, and for each such $\kappa$, we have (say) $10\tau \supseteq \kappa$.  Thus,
\begin{align}\label{ex2}
\notag\|\fE \chi_{\Omega'}\|_{2q}^2 &\lesssim \sum_{j,k} 2^{-(j+k)(1-\frac{2}{q})}\bigg(\sum_{\tau \in \fT_{j,k}}|\Omega \cap 10\tilde{\tau}|^q\bigg)^\frac{1}{q}\\ &\lesssim \sum_{j,k} 2^{-(j+k)(1-\frac{2}{q})}\max_{\tau \in \fT_{j,k}}|\Omega \cap 10\tilde{\tau}|^{1-\frac{1}{q}}|\Omega|^\frac{1}{q}.
\end{align}
Let $J$ be an integer such that $|\pi_{1,3}(\Omega)| \sim 2^{-J}$.  Then, by Fubini's theorem, $|\Omega| \sim 2^{-J-K}$ and
\begin{align} \label{ex3}
\max_{\tau \in \fT_{j,k}}|\Omega \cap 10\tilde{\tau}| \lesssim \min\{2^{-J},2^{-j}\}\min\{2^{-K},2^{-k}\}.
\end{align}
We split the right-hand side of \eqref{ex2} into four parts:  summation over $j,k$ satisfying (i) $j \leq J$, $k \leq K$; (ii) $j \leq J$, $k > K$; (iii) $j > J$, $k \leq K$; (iv) $j > J$, $k > K$.  Each part is estimated simply by applying \eqref{ex3} and summing a geometric series.  We obtain the desired bound in this way.
\end{proof}

\section{An inverse problem related to Proposition \ref{prop1}}\label{sec3}
In this section, we answer quantitatively the following question:  If $\Omega$ extremizes the inequality in Proposition \ref{prop1}, what structure must $\Omega$ have?	
\begin{proposition}\label{prop2}
Suppose that $\Omega = \Omega(K)$ for some $K$, let $J$ be an integer such that $|\Omega| \sim 2^{-J-K}$, and let $\varepsilon := \rC(\Omega)$.  Up to a set of measure zero, there exists a decomposition
\begin{align*}
\Omega = \bigcup_{0 < \delta \lesssim \varepsilon^{1/5}}\Omega_{\delta},
\end{align*}
where the union is taken over dyadic numbers, such that
\begin{enumerate}[nolistsep]
\item[(i)]{$\rC(\Omega_\delta,\Omega) \lesssim \delta^{1/3}$, and}
\item[(ii)]{$\Omega_{\delta} \subseteq \bigcup_{\tau \in \fT_{\delta}} \tilde{\tau}$, where $\fT_{\delta} \subseteq \fT_{J,K}$ with $\#\fT_{\delta} \lesssim \delta^{-C_0}$ for some admissible constant $C_0$.}
\end{enumerate}
\end{proposition}

\begin{proof}[Proof of Proposition \ref{prop2}]
The construction of the sets $\Omega_{\delta}$ consists of five steps.  We will begin by dividing $\Omega$ into sets $\Omega_\alpha^1$ whose $\pi_{1,3}$-projections have constant $\pi_1$-fiber length $\alpha$, respectively.  That simple step enables us to adapt then the decomposition scheme employed in \cite{Stovall}.  We divide each $\Omega_\alpha^1$ into sets $\Omega_{\alpha,\eta}^2$ whose respective projections to the $\zeta_1$-axis are contained in $\eta^{-1}$ intervals in $\fI_{J}$.  In our third step, we divide each $\Omega_{\alpha,\eta}^2$ into sets $\Omega_{\alpha,\eta,\rho}^3$ of constant $\pi_{2,3}$-fiber length $\rho\eta^{-1}2^{-J}$.  To each $\Omega_{\alpha,\eta,\rho}^3$ we may then apply variants of the first two steps wherein the roles of the coordinates $\zeta_1,\zeta_2$ are reversed.  Indeed, were $\pi_{1,3}$ replaced by $\pi_{2,3}$ in Definition \ref{def2}, each $\Omega_{\alpha,\eta,\rho}^3$ would be of the form $\Omega_{\alpha,\eta,\rho}^3(J+\log(\rho^{-1}\eta))$.  In the end, we obtain sets $\Omega_{\alpha,\eta,\rho,\beta,\delta}^5$ whose respective projections to the $\zeta_2$-axis are contained in $\delta^{-1}$ intervals in $\fI_K$.  For fixed $\delta$, we define $\Omega_\delta$ to be the union of the sets $\Omega_{\alpha,\eta,\rho,\beta,\delta}^5$, of which there will be at most $(\log \delta^{-1})^4$ by construction.  Appearing in the argument below, there are of course constants and minor technical adjustments missing from this summary.

\emph{Step 1.} For each dyadic number $0 < \alpha \leq 1$, define
\begin{align*}
\Omega_\alpha^1 &:= \{(\zeta,\sigma) \in \Omega : \fH^1(\pi_1^{-1}(\zeta_1) \cap \pi_{1,3}(\Omega)) \sim \alpha^A\},
\end{align*}
where $A$ is an admissible constant to be chosen momentarily.
\begin{lemma}\label{lem1}
For every $0 < \alpha \leq 1$, we have $\rC(\Omega_\alpha^1,\Omega) \lesssim \alpha$.
\end{lemma}
\begin{proof}[Proof of Lemma \ref{lem1}]
Let $\Omega' \subseteq \Omega_\alpha^1$ be measurable, and let $J_\alpha$ be an integer such that $|\pi_{1,3}(\Omega_\alpha^1)| \sim \alpha^A2^{-J_\alpha}$.  We record the bound
\begin{align}\label{ex3'}
\alpha^A 2^{-J_\alpha} \lesssim 2^{-J}.
\end{align}   
Following the proof of Proposition \ref{prop1}, we have
\begin{align}\label{ex4'}
\|\fE \chi_{\Omega'}\|_{2q}^2 \lesssim \sum_{j,k}2^{-(j+k)(1-\frac{2}{q})}\max_{\tau \in \fT_{j,k}}|\Omega_\alpha^1 \cap 10\tilde{\tau}|^{1-\frac{1}{q}}|\Omega|^\frac{1}{q}.
\end{align}
By Fubini's theorem,
\begin{align}\label{ex5'}
\notag|\Omega_\alpha^1 \cap 10\tilde{\tau}| &\lesssim |\pi_{1,3}(\Omega_\alpha^1 \cap 10\tilde{\tau})|\min\{2^{-K},2^{-k}\}\\
&\lesssim \alpha^A\min\{2^{-J_\alpha},2^{-j}\}\min\{2^{-K},2^{-k}\}
\end{align}
for every $\tau \in \fT_{j,k}$.  As in the proof of Proposition \ref{prop1}, we split the right-hand side of \eqref{ex4'} into four parts:  summation over $j,k$ satisfying (i) $j \leq J_\alpha$, $k \leq K$; (ii) $j \leq J_\alpha$, $k > K$; (iii) $j > J_\alpha$, $k \leq K$; (iv) $j > J_\alpha$, $k > K$.  Using \eqref{ex5'} and \eqref{ex3'}, we bound the sum corresponding to (i) by
\begin{align*}
\sum_{\substack{j \leq J_\alpha\\k \leq K}}2^{-(j+k)(1-\frac{2}{q})}(\alpha^A2^{-J_\alpha-K})^{1-\frac{1}{q}}|\Omega|^\frac{1}{q} &\sim \alpha^{A(1-\frac{1}{q})}2^{-(J_\alpha+K)(2-\frac{3}{q})}|\Omega|^\frac{1}{q}\\
&\lesssim \alpha^{A(\frac{2}{q}-1)}2^{-(J+K)(2-\frac{3}{q})}|\Omega|^\frac{1}{q}\\
&\sim \alpha^{A(\frac{2}{q}-1)}|\Omega|^\frac{2}{q'}.
\end{align*}
Using the same steps, the sum corresponding to (ii) is at most
\begin{align*}
\sum_{\substack{j \leq J_\alpha\\k > K}}2^{-j(1-\frac{2}{q})}2^{-k(2-\frac{3}{q})}\alpha^{A(1-\frac{1}{q})}2^{-J_\alpha(1-\frac{1}{q})}|\Omega|^\frac{1}{q} &\sim \alpha^{A(1-\frac{1}{q})}2^{-(J_\alpha+K)(2-\frac{3}{q})}|\Omega|^\frac{1}{q}\\ &\lesssim \alpha^{A(\frac{2}{q}-1)}|\Omega|^\frac{2}{q'}.
\end{align*}
The sums corresponding to (iii) and (iv) can be handled in essentially the same way, leading to the estimate
\begin{align*}
\|\fE\chi_{\Omega'}\|_{2q} \lesssim \alpha^{A(\frac{1}{q}-\frac{1}{2})}|\Omega|^\frac{1}{q'}.
\end{align*}
We conclude the proof by setting $A := (\frac{1}{q}-\frac{1}{2})^{-1}$.
\end{proof}

\emph{Step 2.} For each $0 < \alpha \leq 1$, let $S_\alpha := \pi_1(\pi_{1,3}(\Omega_\alpha^1))$, and note that $|S_\alpha| \sim 2^{-J_\alpha}$ with $J_\alpha$ as in the proof of Lemma \ref{lem1}.  Given a dyadic number $0 < \eta \leq \alpha$ and a Lebesgue point $\zeta_1$ of $S_\alpha$, let $I_{\alpha,\eta}(\zeta_1)$ be the maximal dyadic interval $I$ such that $\zeta_1 \in I$ and
\begin{align}\label{ex5}
\frac{|I \cap S_\alpha|}{|I|} \geq \eta^B,
\end{align}
where $B$ is an admissible constant to be chosen later; such an interval exists by the Lebesgue differentiation theorem.  Without loss of generality, we assume that $S_\alpha$ is equal to its set of Lebesgue points.  Let
\begin{align*}
T_{\alpha,\eta} := \{\zeta_1 \in S_\alpha : |I_{\alpha,\eta}(\zeta_1)| \geq \eta^B2^{-J_\alpha}\}.
\end{align*}
If $\alpha < \varepsilon$, define $S_{\alpha,\alpha} := T_{\alpha,\alpha}$ and $S_{\alpha,\eta} := T_{\alpha,\eta} \setminus T_{\alpha,2\eta}$ for $\eta < \alpha$, and let
\begin{align*}
\Omega_{\alpha,\eta}^2 := \Omega_\alpha^1 \cap \pi_{1,3}^{-1}(\pi_1^{-1}(S_{\alpha,\eta})).
\end{align*}
For $\varepsilon \leq \alpha \leq 1$, define $S_{\alpha,\varepsilon} := T_{\alpha,\varepsilon}$ and $S_{\alpha,\eta} := T_{\alpha,\eta} \setminus T_{\alpha,2\eta}$ for $\eta < \varepsilon$.  For $\eta \leq \varepsilon$, let
\begin{align*}
\Omega_{\varepsilon,\eta}^2 := \bigcup_{\varepsilon \leq \alpha \leq 1}\tilde{\Omega}_{\alpha,\eta}^2,
\end{align*}
where $\tilde{\Omega}_{\alpha,\eta}^2 := \Omega_{\alpha}^1 \cap \pi_{1,3}^{-1}(\pi_1^{-1}(S_{\alpha,\eta}))$.  
\begin{remark}\label{rmk3.3}
We note that $\Omega_{\alpha,\eta}^2 \subseteq \Omega_{\alpha}^1$ for $\alpha < \varepsilon$ and $\tilde{\Omega}_{\alpha,\eta}^2 \subseteq \Omega_\alpha^1$ for $\varepsilon \leq \alpha \leq 1$, while in general $\Omega_{\varepsilon,\eta}^2$ is not contained in $\Omega_{\varepsilon}^1$.  We do have
\begin{align*}
\Omega = \bigcup_{0 < \alpha \leq 1} \Omega_\alpha^1 =  \bigcup_{0 < \alpha \leq \varepsilon}\bigcup_{0 < \eta \leq \alpha}\Omega_{\alpha,\eta}^2.
\end{align*}
\end{remark}
  
\begin{lemma}\label{lem2}
For every $0 < \eta \leq \alpha \leq \varepsilon$, the set $\Omega_{\alpha,\eta}^2$ is contained in a union of $O(\eta^{-3B-A-1})$ boxes of the form $\tilde{\tau}$, with $\tau \in \fT_{J,0}$, and satisfies $\rC(\Omega_{\alpha,\eta}^2,\Omega) \lesssim \eta^{1/2}$.
\end{lemma}

\begin{proof}[Proof of Lemma \ref{lem2}]
We argue first under the hypothesis that $\alpha < \varepsilon$, then indicate the changes needed when $\alpha = \varepsilon$. By its definition, $S_{\alpha,\eta}$ is covered by dyadic intervals $I$ of length $|I| \gtrsim \eta^B|S_\alpha|$, in each of which $S_\alpha$ has density obeying \eqref{ex5}.  The density of each such $I$ in $S_\alpha$ is
\begin{align*}
\frac{|I \cap S_{\alpha}|}{|S_\alpha|} = \frac{|I \cap S_\alpha|}{|I|}\cdot\frac{|I|}{|S_\alpha|} \gtrsim \eta^{2B}.
\end{align*}
Therefore, if $\fC$ is a minimal-cardinality covering of $S_{\alpha,\eta}$ by these $I$ (consisting necessarily of pairwise disjoint intervals), then $\#\fC \lesssim \eta^{-2B}$.  Moreover, \eqref{ex5} and \eqref{ex3'} imply that  
\begin{align*}
|I| \lesssim \eta^{-B}2^{-J_\alpha} \lesssim \eta^{-B}\alpha^{-A}2^{-J} \leq \eta^{-B-A}2^{-J}
\end{align*}
for every $I \in \fC$.  Thus, $S_{\alpha,\eta}$ is covered by $O(\eta^{-3B-A})$ intervals in $\fI_J$. Since $\alpha < \varepsilon$, it immediately follows that $\Omega_{\alpha,\eta}^2$ is contained in a union of $O(\eta^{-3B-A})$ boxes of the form claimed.  

We turn to the restriction estimate.  If $\eta = \alpha$, the result follows from Lemma \ref{lem1} and Remark \ref{rmk3.3}.  Thus, we may assume that $\eta < \alpha$. We proceed by optimizing the proof of Proposition \ref{prop1}, as in \cite{Stovall}.  Let $\Omega' \subseteq \Omega_{\alpha,\eta}^2$ be measurable.  From the proof of \eqref{ex2}, we see that
\begin{align}\label{ex6}
\|\fE\chi_{\Omega'}\|_{2q}^2 \lesssim \sum_{j,k}2^{-(j+k)(1-\frac{2}{q})}\max_{\tau \in \fT_{j,k}}|\Omega' \cap 10 \tilde{\tau}|^{1-\frac{1}{q}}|\Omega|^\frac{1}{q}.
\end{align}
Fix $\tau \in \fT_{j,k}$.  By Fubini's theorem and the definition of $\Omega_\alpha^1$ (with $\alpha < \varepsilon$), we have
\begin{align}\label{ex7}
\notag|\Omega' \cap 10\tilde{\tau}| &\lesssim |\pi_{1,3}(\Omega' \cap 10\tilde{\tau})|\min\{2^{-K},2^{-k}\}\\
\notag&\lesssim \alpha^A \min\{|\pi_1(\pi_{1,3}(\Omega'))|,|\pi_1(\pi_{1,3}(10\tilde{\tau}))|\}\min\{2^{-K},2^{-k}\}\\
&\lesssim \alpha^A\min\{2^{-J_\alpha},2^{-j}\}\min\{2^{-K},2^{-k}\}.
\end{align}
For certain $j$, the definition of $\Omega_{\alpha,\eta}^2$ leads to a better estimate.  We claim that if $|j-J_\alpha| < \frac{B}{4}\log\eta^{-1}$, then
\begin{align}\label{ex8}
|\Omega' \cap 10\tilde{\tau}| \lesssim \eta^{\frac{3B}{4}}\alpha^A\min\{2^{-J_\alpha},2^{-j}\}\min\{2^{-K},2^{-k}\}.
\end{align}
Fix such a $j$.  Note that $10\tau$ is contained in a union of four tiles $\kappa$ in $\fT_{j-4,k-4}$, so it suffices to prove \eqref{ex8} with $\kappa$ in place of $10\tau$.  Let $\kappa =: I_{j-4} \times I_{k-4}$, where $I_{j-4} \in \fI_{j-4}$ and $I_{k-4} \in \fI_{k-4}$.  We have
\begin{align*}
|I_{j-4}| = 2^{-j+4} \geq 16\eta^\frac{B}{4}2^{-J_\alpha} \geq (2\eta)^B 2^{-J_\alpha},
\end{align*}
provided $\eta$ is sufficiently small.  Suppose that $I_{j-4} \cap S_{\alpha,\eta} \neq \emptyset$.  Then there exists $\zeta_1 \in I_{j-4}$ such that $\zeta_1 \notin T_{\alpha,2\eta}$, whence
\begin{align*}
|I_{\alpha,2\eta}(\zeta_1)| < (2\eta)^B2^{-J_\alpha} \leq |I_{j-4}|.
\end{align*}
Consequently, by the maximality of $I_{\alpha,2\eta}(\zeta_1)$ and the fact that $2^{-j} \leq \eta^{-\frac{B}{4}}2^{-J_\alpha}$, we have
\begin{align*}
|I_{j-4} \cap S_{\alpha,\eta}| \leq |I_{j-4} \cap S_\alpha| \leq (2\eta)^B|I_{j-4}| = 16(2\eta)^B2^{-j} \lesssim \eta^\frac{3B}{4}\min\{2^{-J_\alpha},2^{-j}\}.
\end{align*}
Thus, by Fubini's theorem,
\begin{align*}
|\Omega' \cap \tilde{\kappa}| \lesssim \alpha^A|S_{\alpha,\eta} \cap I_{j-4}|\min\{2^{-K},2^{-k}\} \lesssim \eta^\frac{3B}{4}\alpha^A\min\{2^{-J_\alpha},2^{-j}\}\min\{2^{-K},2^{-k}\},
\end{align*}
as claimed.

Now, to bound \eqref{ex6}, we split the sum into eight parts determined by the conditions (a) $k \leq K$, (b) $k > K$ and (i) $j \leq J_\alpha - \frac{B}{4}\log\eta^{-1}$, (ii) $J_\alpha - \frac{B}{4}\log\eta^{-1} < j \leq J_\alpha$, (iii) $J_\alpha < j <J_\alpha + \frac{B}{4}\log\eta^{-1}$, (iv) $J_\alpha + \frac{B}{4}\log\eta^{-1} \leq j$.  In each case, we use \eqref{ex8} if it applies, otherwise \eqref{ex7}. Summing geometric series and using \eqref{ex3'} and the fact that $|\Omega| \sim 2^{-J-K}$, it is straightforward to deduce the bound
\begin{align*}
\|\fE\chi_{\Omega'}\|_{2q} \lesssim \eta^{B'}|\Omega|^\frac{1}{q'},
\end{align*}
where $B'$ is an admissible constant determined by $B$.  We may choose $B$ so that $B' = 1$; this better-than-required exponent will be utilized in the next paragraph.

Suppose now that $\alpha = \varepsilon$.  For $\eta < \varepsilon$, the preceding arguments work equally well with $\Omega_{\alpha,\eta}^2$ replaced by $\tilde{\Omega}_{\alpha',\eta}^2$, where $\varepsilon \leq \alpha' \leq 1$.  In particular, each such $\tilde{\Omega}_{\alpha',\eta}^2$ is contained in a union of $O(\eta^{-3B-A})$ boxes $\tilde{\tau}$, with $\tau \in \fT_{J,0}$, and satisfies $\rC(\tilde{\Omega}_{\alpha',\eta}^2,\Omega) \lesssim \eta$.  The case $\eta = \varepsilon$ is similar, but with the bound $\rC(\tilde{\Omega}_{\alpha',\varepsilon}^2,\Omega) \lesssim \varepsilon$ following directly from the definition of $\varepsilon$. Since the number of sets $\tilde{\Omega}_{\alpha',\eta}^2$ is $O(\log\varepsilon^{-1}) = O(\eta^{-1/2})$ and their union is $\Omega_{\varepsilon,\eta}^2$, the lemma holds for $\alpha = \varepsilon$ as well.
\end{proof}

\emph{Step 3.} For dyadic $0 < \eta \leq \alpha \leq \varepsilon$ and $0 < \rho \lesssim \eta^{1/5}$, define
\begin{align*}
\Omega_{\alpha,\eta,\rho}^3 &:= \{(\zeta,\sigma) \in \Omega_{\alpha,\eta}^2 : \fH^{1}(\pi_{2,3}^{-1}(\zeta_2,\sigma) \cap \Omega_{\alpha,\eta}^2) \sim \rho^{5C}\eta^{-3B-A-1-C}2^{-J}\},
\end{align*} 
where $C$ is an admissible constant to be chosen later.  Lemma \ref{lem2} implies that $\fH^1(\pi_{2,3}^{-1}(\zeta_2,\sigma) \cap \Omega_{\alpha,\eta}^2) \lesssim \eta^{-3B-A-1}2^{-J}$ for every $(\zeta,\sigma) \in \Omega_{\alpha,\eta}^2$.  Thus,
\begin{align*}
\Omega_{\alpha,\eta}^2 = \bigcup_{0 < \rho \lesssim \eta^{1/5}} \Omega_{\alpha,\eta,\rho}^3.
\end{align*}

\begin{lemma}\label{lem3}
For every $0 < \eta \leq \alpha \leq \varepsilon$ and $0 < \rho \lesssim \eta^{1/5}$, we have $\rC(\Omega_{\alpha,\eta,\rho}^3,\Omega) \lesssim \rho$.
\end{lemma}

\begin{proof}[Proof of Lemma \ref{lem3}]
If $\rho^{5C}\eta^{-3B-A-1-C} \geq \rho^{2C}$, then by Lemma \ref{lem2}, we have
\begin{align*}
\rC(\Omega_{\alpha,\eta,\rho}^3,\Omega) \lesssim \eta^\frac{1}{2} \leq \rho^\frac{3C}{2(3B+A+1+C)} \lesssim \rho
\end{align*}
for $C$ chosen sufficiently large.  Thus, we may assume that $\rho^{5C}\eta^{-3B-A-1-C} \leq \rho^{2C}$. Given a measurable set $\Omega' \subseteq \Omega_{\alpha,\eta,\rho}^3$ and $\tau \in \fT_{j,k}$, the set $\Omega' \cap 10\tilde{\tau}$ has $\pi_{1,3}$- and $\pi_{2,3}$-fibers of length at most $\min\{2^{-K},2^{-k}\}$ and $\min\{\rho^{2C}2^{-J},2^{-j}\}$, respectively, and it has $\pi_{1,3}$- and $\pi_{2,3}$-projections of measure at most $\min\{2^{-J},2^{-j}\}$ and $2^{-k}$, respectively.  Therefore, by Fubini's theorem,
\begin{align}\label{ex18'}
|\Omega' \cap 10\tilde{\tau}| \lesssim \min\{2^{-J-K},2^{-j-K},2^{-j-k},\rho^{2C}2^{-J-k}\}.
\end{align}
Following \cite{Stovall}, we define
\begin{align*}
R_1 &:= \{(j,k) : J- C\log\rho^{-1} \geq j,~K \geq k\} \cup \{(j,k) : J \geq j,~K-C\log\rho^{-1} \geq k\}\\
R_2 &:= \{(j,k) : j \geq J+C\log\rho^{-1},~K \geq k\} \cup \{(j,k) : j \geq J,~K-C\log\rho^{-1} \geq k\}\\
R_3 &:= \{(j,k) : j \geq J+C\log\rho^{-1},~k \geq K\} \cup \{(j,k) : j \geq J,~k \geq K+C\log\rho^{-1}\}\\
R_4 &:= \{(j,k) : J + C\log\rho^{-1} \geq j,~k+C\log\rho^{-1} \geq K\}.
\end{align*}
Each $(j,k)$ belongs to some $R_i$, $1 \leq i \leq 4$, so by \eqref{ex6} and \eqref{ex18'}, we have
\begin{align*}
\|\fE\chi_{\Omega'}\|_{2q}^2 \lesssim \sum_{(j,k) \in R_1}2^{-(j+k)(1-\frac{2}{q})}2^{-(J+K)(1-\frac{1}{q})}|\Omega|^\frac{1}{q} + \sum_{(j,k) \in R_2}2^{-(j+k)(1-\frac{2}{q})}2^{-(j+K)(1-\frac{1}{q})}|\Omega|^\frac{1}{q}\\
+ \sum_{(j,k) \in R_3}2^{-(j+k)(1-\frac{2}{q})}2^{-(j+k)(1-\frac{1}{q})}|\Omega|^\frac{1}{q} + \sum_{(j,k) \in R_4}2^{-(j+k)(1-\frac{2}{q})}\rho^{2C(1-\frac{1}{q})}2^{-(J+k)(1-\frac{1}{q})}|\Omega|^\frac{1}{q}.
\end{align*}
Summing these geometric series leads to the bound $\|\fE\chi_{\Omega'}\|_{2q} \lesssim \rho^{C'}|\Omega|^{1/q'}$, where $C'$ is an admissible constant determined by $C$; increasing $C$ if necessary, we can make $C' \geq 1$.
\end{proof}

As indicated above, the final two steps of our construction are variants of the first two, wherein the roles of the coordinates $\zeta_1,\zeta_2$ are reversed.  Below, we briefly explain how the argument in Steps 1 and 2 transfers, without rewriting all the details.  In short, $\Omega_{\alpha,\eta,\rho}^3$ has constant $\pi_{2,3}$-fiber length by construction and thus may replace $\Omega$, and $\rho$ may replace $\varepsilon$ by Lemma \ref{lem3}.

\emph{Step 4.} For each dyadic number $0 < \beta \leq 1$, define
\begin{align*}
\Omega_{\alpha,\eta,\rho,\beta}^4 := \{(\zeta,\sigma) \in \Omega_{\alpha,\eta,\rho}^3 : \fH^1(\pi_1^{-1}(\zeta_2) \cap \pi_{2,3}(\Omega_{\alpha,\eta,\rho}^3)) \sim \beta^A\}.
\end{align*}
\begin{lemma}\label{lem4}
For every $0 < \beta \leq 1$, $0 < \eta \leq \alpha \leq \varepsilon$, and $0 < \rho \lesssim \eta^{1/5}$, we have $\rC(\Omega_{\alpha,\eta,\rho,\beta}^4,\Omega) \lesssim \beta$.
\end{lemma}
\begin{proof}[Proof of Lemma \ref{lem4}]
Since $\Omega_{\alpha,\eta,\rho}^3$ has constant $\pi_{2,3}$-fiber length, we can imitate the proof of Lemma \ref{lem1} to show that $\beta \gtrsim \rC(\Omega_{\alpha,\eta,\rho,\beta}^4, \Omega_{\alpha,\eta,\rho}^3) \geq \rC(\Omega_{\alpha,\eta,\rho,\beta}^4,\Omega)$.
\end{proof}

\emph{Step 5.} For each $0 < \beta \leq 1$, let $S_{\alpha,\eta,\rho,\beta} := \pi_1(\pi_{2,3}(\Omega_{\alpha,\eta,\rho,\beta}^4))$, and let $K_{\alpha,\eta,\rho,\beta}$ be an integer such that $|S_{\alpha,\eta,\rho,\beta}| \sim 2^{-K_{\alpha,\eta,\rho,\beta}}$.  Given a dyadic number $0 < \delta \leq \beta$ and a Lebesgue point $\zeta_2$ of $S_{\alpha,\eta,\rho,\beta}$, let $I_{\alpha,\eta,\rho,\beta,\delta}(\zeta_2)$ be the maximal dyadic interval $I$ such that $\zeta_2 \in I$ and
\begin{align*}
\frac{|I \cap S_{\alpha,\eta,\rho,\beta}|}{|I|} \geq \delta^B.
\end{align*}
As before, we may assume that $S_{\alpha,\eta,\rho,\beta}$ is equal to its set of Lebesgue points.  Let
\begin{align*}
T_{\alpha,\eta,\rho,\beta,\delta} := \{\zeta_2 \in S_{\alpha,\eta,\rho,\beta} : |I_{\alpha,\eta,\rho,\beta,\delta}(\zeta_2)| \geq \delta^B2^{-K_{\alpha,\eta,\rho,\beta}}\}.
\end{align*}
If $\beta < \rho$, define $S_{\alpha,\eta,\rho,\beta,\beta} := T_{\alpha,\eta,\rho,\beta,\beta}$ and $S_{\alpha,\eta,\rho,\beta,\delta} := T_{\alpha,\eta,\rho,\beta,\delta} \setminus T_{\alpha,\eta,\rho,\beta,2\delta}$ for $\delta < \beta$, and let
\begin{align*}
\Omega_{\alpha,\eta,\rho,\beta,\delta}^5 := \Omega_{\alpha,\eta,\rho,\beta}^4 \cap \pi_{2,3}^{-1}(\pi_1^{-1}(S_{\alpha,\eta,\rho,\beta,\delta})).
\end{align*}
If $\rho \leq \beta \leq 1$, define $S_{\alpha,\eta,\rho,\beta,\rho} := T_{\alpha,\eta,\rho,\beta,\rho}$ and $S_{\alpha,\eta,\rho,\beta,\delta} := T_{\alpha,\eta,\rho,\beta,\delta} \setminus T_{\alpha,\eta,\rho,\beta,2\delta}$ for $\delta < \rho$.  For $\delta \leq \rho$, let
\begin{align*}
\Omega_{\alpha,\eta,\rho,\rho,\delta}^5 := \bigcup_{\rho \leq \beta \leq 1} \tilde{\Omega}_{\alpha,\eta,\rho,\beta,\delta}^5,
\end{align*}
where $\tilde{\Omega}_{\alpha,\eta,\rho,\beta,\delta}^5 := \Omega_{\alpha,\eta,\rho,\beta}^4 \cap \pi_{2,3}^{-1}(\pi_1^{-1}(S_{\alpha,\eta,\rho,\beta,\delta}))$.

Admittedly, the subscripts have become awkward.  However, all we have done is repeated Step 2, replacing  $\Omega_\alpha^1$ and $\varepsilon$ by $\Omega_{\alpha,\eta,\rho,\beta}^4$ and $\rho$, respectively, and projecting onto the $\zeta_2$-axis instead of the $\zeta_1$-axis.  We note that
\begin{align*}
\Omega_{\alpha,\eta,\rho}^3 = \bigcup_{0 < \beta \leq \rho}\bigcup_{0 < \delta \leq \beta}\Omega_{\alpha,\eta,\rho,\beta,\delta}^5.
\end{align*}

\begin{lemma}\label{lem5}
For every $0 < \eta \leq \alpha \leq \varepsilon$ and $0 < \delta \leq \beta \leq \rho \lesssim \eta^{1/5}$, the set $\Omega_{\alpha,\eta,\rho,\beta,\delta}^5$ is contained in a union of $O(\delta^{-18B-6A-5C-6})$ boxes of the form $\tilde{\tau}$, with $\tau \in \fT_{J,K}$, and satisfies $\rC(\Omega_{\alpha,\eta,\rho,\beta,\delta}^5,\Omega) \lesssim \delta^{1/2}$.
\end{lemma}
\begin{proof}[Proof of Lemma \ref{lem5}]
Let $K_{\alpha,\eta,\rho}$ be an integer such that $|\pi_{2,3}(\Omega_{\alpha,\eta,\rho}^3)| \sim 2^{-K_{\alpha,\eta,\rho}}$.  Imitating the proof of Lemma \ref{lem2}, we can show that $\Omega_{\alpha,\eta,\rho,\beta,\delta}^5$ is covered by $O(\delta^{-3B-A-1})$ boxes of the form $\tilde{\tau}$, where $\tau \in \fT_{0,K_{\alpha,\eta,\rho}}$.  Since $\Omega_{\alpha,\eta,\rho}^3$ has $\pi_{2,3}$-fibers of length $\rho^{5C}\eta^{-3B-A-1-C}2^{-J}$ and volume at most $2^{-J-K}$, it follows that $2^{-K_{\alpha,\eta,\rho}} \lesssim \rho^{-5C}2^{-K}$.  Thus, $\Omega_{\alpha,\eta,\rho,\beta,\delta}^5$ is covered by $O(\rho^{-5C}\delta^{-3B-A-1}) = O(\delta^{-3B-A-5C-1})$ boxes $\tilde{\tau}$, with $\tau \in \fT_{0,K}$.  Since $\Omega_{\alpha,\eta,\rho,\beta,\delta}^5 \subseteq \Omega_{\alpha,\eta}^2$ and $\eta \gtrsim \delta^5$, Lemma \ref{lem2} now implies that $\Omega_{\alpha,\eta,\rho,\beta,\delta}^5$ is covered by $O(\delta^{-18B-6A-5C-6})$ boxes $\tilde{\tau}$, with $\tau \in \fT_{J,K}$.

To obtain the restriction estimate, we can adapt the proof of Lemma \ref{lem2}.  
\end{proof}

Finally, we are equipped to finish the proof of Proposition \ref{prop2}. We have
\begin{align*}
\Omega = \bigcup_{0 < \alpha \leq \varepsilon}\bigcup_{0 < \eta \leq \alpha}\bigcup_{0 < \rho \lesssim \eta^{1/5}}\bigcup_{0 < \beta \leq \rho}\bigcup_{0 < \delta \leq \beta}\Omega_{\alpha,\eta,\rho,\beta,\delta}^5 = \bigcup_{0 < \delta \lesssim \varepsilon^{1/5}}\Omega_\delta,
\end{align*}
where
\begin{align*}
\Omega_\delta := \bigcup_{\delta \leq \beta \lesssim \varepsilon^{1/5}}\bigcup_{\beta \leq \rho \lesssim \varepsilon^{1/5}}\bigcup_{\rho^5 \lesssim \eta \leq \varepsilon}\bigcup_{\eta \leq \alpha \leq \varepsilon}\Omega_{\alpha,\eta,\rho,\beta,\delta}^5.
\end{align*}
Since for fixed $\delta$ there are $O((\log \delta^{-1})^4)$ sets $\Omega_{\alpha,\eta,\rho,\beta,\delta}^5$, properties (i) and (ii) in the proposition follow from Lemma \ref{lem5}.

\end{proof}

\section{Extensions of near unions of boxes}\label{sec4}
For each $K$, let $J(K)$ be an integer such that $|\Omega(K)| \sim 2^{-J(K)-K}$.  For each dyadic number $\varepsilon$, let $\fK(\varepsilon)$ denote the collection of all integers $K \geq 0$ for which $\varepsilon = \rC(\Omega(K))$.  For each $K \in \fK(\varepsilon)$, Proposition \ref{prop2} gives a decomposition $\Omega(K) = \bigcup_{0 < \delta \lesssim \varepsilon^{1/5}} \Omega(K)_\delta$ such that for each $\delta$, we have $\Omega(K)_\delta \subseteq \bigcup_{\tau \in \fT(K)_\delta}\tilde{\tau}$ for some $\fT(K)_\delta \subseteq \fT_{J(K),K}$ with $\#\fT(K)_\delta \lesssim \delta^{-C_0}$.

\begin{lemma}\label{lem7}
For every $0 < \delta \lesssim \varepsilon^{1/5}$, we have
\begin{align*}
\bigg\|\sum_{K \in \fK(\varepsilon)}\fE\chi_{\Omega(K)_\delta}\bigg\|_{2q}^{2q} \lesssim (\log \delta^{-1})^{2q}\sum_{K \in \fK(\varepsilon)}\|\fE\chi_{\Omega(K)_\delta}\|_{2q}^{2q} + \delta|\Omega|^\frac{2q}{q'}.
\end{align*}
\end{lemma}
\begin{proof}[Proof of Lemma \ref{lem7}]
Let $A$ be an admissible constant to be chosen later, and divide $\fK(\varepsilon)$ into $O(\log\delta^{-1})$ subsets $\fK$ such that each is $A\log\delta^{-1}$-separated.  It suffices to prove that
\begin{align*}
\bigg\|\sum_{K \in \fK}\fE\chi_{\Omega(K)_\delta}\bigg\|_{2q}^{2q} \lesssim \sum_{K \in \fK}\|\fE\chi_{\Omega(K)_\delta}\|_{2q}^{2q} + \delta^2|\Omega|^\frac{2q}{q'}
\end{align*}
for each $\fK$.  Since $q < 2$, we have
\begin{align}\label{ex17'}
\notag \bigg\|\sum_{K \in \fK} \fE\chi_{\Omega(K)_\delta}\bigg\|_{2q}^{2q} &= \int\bigg\vert\sum_{\boldsymbol{K} \in \fK^4}\prod_{i=1}^4 \fE\chi_{\Omega(K_i)_\delta}\bigg\vert^\frac{q}{2}\\ &\lesssim \sum_{K \in \fK}\|\fE\chi_{\Omega(K)_\delta}\|_{2q}^{2q} + \sum_{ \boldsymbol{K} \in \fK^4 \setminus D(\fK^4)}\bigg\|\prod_{i=1}^4\fE\chi_{\Omega(K_i)_\delta}\bigg\|_{\frac{q}{2}}^\frac{q}{2},
\end{align}
where $D(\fK^4) := \{\boldsymbol{K} \in \fK^4 : K_1 = \cdots = K_4\}$.  To control the latter sum, we have the following lemma.

\begin{lemma}\label{lem8}
For all $K,K' \in \fK$, we have
\begin{align}\label{ex13}
\|\fE\chi_{\Omega(K)_\delta}\fE\chi_{\Omega(K')_\delta}\|_q \lesssim 2^{-c_0|K-K'|}\max\{|\Omega(K)|,|\Omega(K')|\}^\frac{2}{q'}
\end{align}
for some admissible constant $c_0$.
\end{lemma}
\begin{proof}[Proof of Lemma \ref{lem8}]
By the Cauchy--Schwarz inequality and Proposition \ref{prop1},
\begin{align*}
\|\fE\chi_{\Omega(K)_\delta}\fE\chi_{\Omega(K')_\delta}\|_q \lesssim |\Omega(K)|^\frac{1}{q'}|\Omega(K')|^\frac{1}{q'}.
\end{align*}
For $J := J(K)$ and $J' := J(K')$, we have
\begin{align*}
|\Omega(K)|^\frac{1}{q'}|\Omega(K')|^\frac{1}{q'} \lesssim 2^{-\frac{|K-K'|}{q'}}\max\{|\Omega(K)|,|\Omega(K')|\}^\frac{2}{q'}
\end{align*}
whenever either (i) $K = K'$, (ii) $J = J'$, (iii) $J < J'$ and $K < K'$, or (iv) $J > J'$ and $K > K'$; in these cases, \eqref{ex13} follows immediately. 

Thus, by symmetry, it suffices to prove \eqref{ex13} for $K < K'$ and $J > J'$.  By the bound $\#(\fT(K)_\delta \times \fT(K')_\delta) \lesssim \delta^{-2C_0}$ and the separation condition on $\fK$ (with $A$ sufficiently large), it suffices to prove that
\begin{align}\label{ex14}
\|\fE\chi_{\Omega(K)_\delta \cap \tilde{\tau}}\fE\chi_{\Omega(K')_\delta \cap \tilde{\kappa}}\|_q \lesssim 2^{-c|K-K'|}\max\{|\Omega(K)|,|\Omega(K')|\}^\frac{2}{q'}
\end{align}
for all $\tau \in \fT(K)_\delta$, $\kappa \in \fT(K')_\delta$, and some admissible constant $c$.

Fix two such tiles $\tau,\kappa$, and note that $\tau$ must be taller than $\kappa$ and $\kappa$ wider than $\tau$.  By translation, we may assume that the $\zeta_2$- and $\zeta_1$-axes intersect the centers of $\tau$ and $\kappa$, respectively.  Define
\begin{align*}
\tau_k := \begin{cases} \tau \cap \{\zeta : |\zeta_2| \sim 2^{-k}\}, &k < K',\\ \tau \cap \{\zeta : |\zeta_2| \lesssim 2^{-{K'}}\}, &k = K'\end{cases} \quad \text{and} \quad \kappa_j := \begin{cases} \kappa \cap \{\zeta : |\zeta_1| \sim 2^{-j}\}, &j < J,\\ \kappa \cap \{\zeta : |\zeta_1| \lesssim 2^{-{J}}\}, &j = J\end{cases}, 
\end{align*}
so that
\begin{align*}
\tau = \bigcup_{k=0}^{K'}\tau_k \quad\quad\text{and}\quad\quad \kappa = \bigcup_{j=0}^J \kappa_j.
\end{align*}
By the two-parameter Littlewood--Paley square function estimate and fact that $q < 2$, we have
\begin{align}\label{ex16'}
\notag\|\fE\chi_{\Omega(K)_\delta \cap \tilde{\tau}}\fE\chi_{\Omega(K')_\delta \cap \tilde{\kappa}}\|_q^q &\lesssim \int\bigg(\sum_{k=0}^{K'}\sum_{j=0}^J|\fE\chi_{\Omega(K)_\delta \cap \tilde{\tau}_k}\fE\chi_{\Omega(K')_\delta \cap \tilde{\kappa}_j}|^2\bigg)^\frac{q}{2}\\
&\lesssim \sum_{k=0}^{K'}\sum_{j=0}^J\|\fE\chi_{\Omega(K)_\delta \cap \tilde{\tau}_k}\fE\chi_{\Omega(K')_\delta \cap \tilde{\kappa}_j}\|_q^q,
\end{align}
where $\tilde{\tau}_k := \tau_k \times [1,2]$ and $\tilde{\kappa}_j := \kappa_j \times [1,2]$.  We first sum the terms with $k = K'$.  By the Cauchy--Schwarz inequality and Proposition \ref{prop1}, we have
\begin{align*}
\sum_{j=0}^J\|\fE\chi_{\Omega(K)_\delta \cap \tilde{\tau}_{K'}}\fE\chi_{\Omega(K')_\delta \cap \tilde{\kappa}_j}\|_q^q \lesssim \sum_{j=0}^J|\tilde{\tau}_{K'}|^\frac{q}{q'}|\tilde{\kappa}_j|^\frac{q}{q'}.
\end{align*}
Since $\kappa$ has width $2^{-{J'}}$, there are at most two nonempty $\kappa_j$ with $j \leq J'$.  This fact and the bound
\begin{align}\label{ex15}
|\tilde{\kappa}_j| \leq \min\{2^{-(j-J')},1\}|\tilde{\kappa}|
\end{align}
imply that $\sum_{j=0}^J|\tilde{\kappa}_j|^\frac{q}{q'} \lesssim |\kappa|^\frac{q}{q'}$.  Since $|\tilde{\tau}_{K'}| \lesssim 2^{-(K'-K)}|\tilde{\tau}|$, $|\tilde{\tau}| \sim |\Omega(K)|$, and $|\tilde{\kappa}| \sim |\Omega(K')|$, we altogether have
\begin{align*}
\sum_{j=0}^J\|\fE\chi_{\Omega(K)_\delta \cap \tilde{\tau}_{K'}}\fE\chi_{\Omega(K')_\delta \cap \tilde{\kappa}_j}\|_q^q \lesssim 2^{-(K'-K)\frac{q}{q'}}|\Omega(K)|^\frac{q}{q'}|\Omega(K')|^\frac{q}{q'},
\end{align*}
which is acceptable.  A similar argument shows that
\begin{align*}
\sum_{k=0}^{K'}\|\fE\chi_{\Omega(K)_\delta \cap \tilde{\tau}_k}\chi_{\Omega(K')_\delta \cap \tilde{\kappa}_J}\|_q^q &\lesssim 2^{-(J-J')\frac{q}{q'}}|\Omega(K)|^\frac{q}{q'}|\Omega(K')|^\frac{q}{q'}\\ &\sim 2^{-(K'-K)\frac{q}{q'}}|\Omega(K)|^\frac{2q}{q'}.
\end{align*}

We now consider the terms with $k < K'$ and $j < J$.  In this case, $\tau_k$ is a subset of four tiles in $\fT_{J,\max\{K,k\}}$ and $\kappa_j$ is a subset of four tiles in $\fT_{\max\{J',j\},K'}$.  Moreover, these tiles are separated by a distance of $2^{-k}$ and $2^{-j}$ in the vertical and horizontal directions, respectively.  Thus, by Theorem \ref{thm2} (rescaled, as in the proof of Proposition \ref{prop1}),
\begin{align*}
\|\fE\chi_{\Omega(K)_\delta \cap \tilde{\tau}_k}\fE\chi_{\Omega(K')_\delta \cap \tilde{\kappa}_j}\|_q \lesssim 2^{-(j+k)(1-\frac{2}{q})}|\Omega(K) \cap \tilde{\tau}_k|^\frac{1}{2}|\Omega(K') \cap \tilde{\kappa}_j|^\frac{1}{2}.
\end{align*}
Using \eqref{ex15} and the analogous bound for $|\tilde{\tau}_k|$, we now get
\begin{align*}
\sum_{k=0}^{K'-1}\sum_{j=0}^{J-1}\|\fE\chi_{\Omega(K)_\delta \cap \tilde{\tau}_k}\fE\chi_{\Omega(K')_\delta \cap \tilde{\kappa}_j}\|_q^q &\lesssim 2^{-(J'-K)(q-2)}|\tilde{\tau}|^\frac{q}{2}|\tilde{\kappa}|^\frac{q}{2}\\ &\sim 2^{(J'-J + K-K')(1-\frac{q}{2})}|\Omega(K)|^\frac{q}{q'}|\Omega(K')|^\frac{q}{q'}.
\end{align*}
By the relations $K < K'$, $J > J'$ and \eqref{ex16'}, we have now proved \eqref{ex14}.
\end{proof}

Returning to the proof of Lemma \ref{lem7}, we consider the second sum in \eqref{ex17'}.  Given $\boldsymbol{K} \in \fK^4 \setminus D(\fK^4)$, let $p(\boldsymbol{K}) = (p_i(\boldsymbol{K}))_{i=1}^4$ be a permutation of $\boldsymbol{K}$ such that $|\Omega(p_1(\boldsymbol{K}))|$ is maximal among $|\Omega(K_i)|$, $1 \leq i \leq 4$, and such that $|K_i - K_j| \leq 2|p_1(\boldsymbol{K}) - p_2(\boldsymbol{K})|$ for all $1 \leq i,j \leq 4$.  Then by the Cauchy--Schwarz inequality, Lemma \ref{lem8}, the separation condition on $\fK$, the fact that $q' < 2q$, and choosing $A$ sufficiently large, we get
\begin{align*}
\sum_{\boldsymbol{K} \in \fK^4 \setminus D(\fK^4)}\bigg\|\prod_{i=1}^4\fE\chi_{\Omega(K_i)_\delta}\bigg\|_\frac{q}{2}^\frac{q}{2} &\lesssim \sum_{\substack{\boldsymbol{K} \in \fK^4 \setminus D(\fK^4)\\\boldsymbol{K} = p(\boldsymbol{K})}} 2^{-c_0|p_1(\boldsymbol{K})-p_2(\boldsymbol{K})|}|\Omega(p_1(\boldsymbol{K}))|^\frac{2q}{q'}\\
&\lesssim \sum_{K_1 \in \fK} \sum_{K_2 \in \fK} |K_1 - K_2|^2 2^{-c_0|K_1 - K_2|}|\Omega(K_1)|^\frac{2q}{q'}\\
&\lesssim \delta^{\frac{c_0A}{2}}\sum_{K_1 \in \fK}|\Omega(K_1)|^\frac{2q}{q'}\\
&\lesssim \delta^2|\Omega|^\frac{2q}{q'}.
\end{align*}
\end{proof}

\section{Proof of Theorem \ref{thm1}}\label{sec5}
In this final section, we prove our main result.  We recall our setup:  For $\Omega \subseteq [-1,1]^2 \times [1,2]$ a measurable set, we have divided $\Omega$ into sets $\Omega(K)$ of constant fiber length $2^{-K}$, partitioned the indices $K$ into sets $\fK(\varepsilon)$ according to the value of $\varepsilon := \rC(\Omega(K))$, and decomposed each $\Omega(K)$ into near unions of boxes $\Omega(K)_\delta$ for $0 < \delta \lesssim \varepsilon^{1/5}$.  Thus,
\begin{align*}
\Omega = \bigcup_{0 < \varepsilon \lesssim 1}\bigcup_{0 < \delta \lesssim \varepsilon^{1/5}}\bigcup_{K \in \fK(\varepsilon)}\Omega(K)_\delta.
\end{align*}
(Actually, there may be $K$ such that $\rC(\Omega(K)) = 0$; however, those terms contribute nothing to the left-hand side below.)
\begin{proof}[Proof of Theorem \ref{thm1}]
By the triangle inequality, Lemma \ref{lem7}, Proposition \ref{prop2}, and the fact that $q' < 2q$, we have
\begin{align*}
\|\fE\chi_\Omega\|_{2q} &\leq \sum_{0 < \varepsilon \lesssim 1}\sum_{0 < \delta \lesssim \varepsilon^{1/5}}\bigg\|\sum_{K \in \fK(\varepsilon)}\fE\chi_{\Omega(K)_\delta}\bigg\|_{2q}\\
&\lesssim \sum_{0 < \varepsilon \lesssim 1}\sum_{0 < \delta \lesssim \varepsilon^{1/5}}\bigg((\log\delta^{-1})^{2q}\sum_{K \in \fK(\varepsilon)}\|\fE\chi_{\Omega(K)_\delta}\|_{2q}^{2q} + \delta|\Omega|^\frac{2q}{q'}\bigg)^\frac{1}{2q}\\
&\lesssim \Bigg[\sum_{0 < \varepsilon \lesssim 1}\sum_{0 < \delta \lesssim \varepsilon^{1/5}} (\log\delta^{-1})\delta^\frac{1}{3}\bigg(\sum_{K \in \fK(\varepsilon)}|\Omega(K)|^\frac{2q}{q'}\bigg)^\frac{1}{2q}\Bigg] + |\Omega|^\frac{1}{q'}\\
&\lesssim \Bigg[\sum_{0 < \varepsilon \lesssim 1}\sum_{0 < \delta \lesssim \varepsilon^{1/5}} (\log\delta^{-1})\delta^\frac{1}{3} |\Omega|^\frac{1}{q'}\Bigg] + |\Omega|^\frac{1}{q'}\\
&\lesssim |\Omega|^\frac{1}{q'},
\end{align*}
proving \eqref{ex1}.
\end{proof}

\end{document}